\documentclass[12pt, oneside,reqno]{amsart}
\usepackage{bbm}
\usepackage{mathtools}
\usepackage{amsmath, amsfonts, amssymb}
\usepackage{eucal}
\usepackage{mathrsfs}
\usepackage{latexsym}
\usepackage{cite}
\usepackage{xcolor}
\definecolor{refkey}{gray}{.45}
\definecolor{labelkey}{gray}{.45}

\usepackage{amsmath,amssymb}
\makeatletter
\newsavebox\myboxA
\newsavebox\myboxB
\newlength\mylenA

\newcommand*\xoverline[2][0.75]{%
    \sbox{\myboxA}{$\m@th#2$}%
    \setbox\myboxB\null% Phantom box
    \ht\myboxB=\ht\myboxA%
    \dp\myboxB=\dp\myboxA%
    \wd\myboxB=#1\wd\myboxA% Scale phantom
    \sbox\myboxB{$\m@th\overline{\copy\myboxB}$}%  Overlined phantom
    \setlength\mylenA{\the\wd\myboxA}%   calc width diff
    \addtolength\mylenA{-\the\wd\myboxB}%
    \ifdim\wd\myboxB<\wd\myboxA%
       \rlap{\hskip 0.5\mylenA\usebox\myboxB}{\usebox\myboxA}%
    \else
        \hskip -0.5\mylenA\rlap{\usebox\myboxA}{\hskip 0.5\mylenA\usebox\myboxB}%
    \fi}
\makeatother

\newcommand{\C}{\mathcal C}
\newcommand\wtilde[1]{\overset{\lower.4ex\hbox{$\scriptstyle \sim$}}{#1}}
\newcommand\wst[1]{\overset{\lower.5ex\hbox{$\scriptscriptstyle \sim$}}{#1}}
\newcommand{\blb}{\raise.3ex\hbox{$\scriptstyle \pmb \lbrack$}}
\newcommand{\sblb}{\raise.1ex\hbox{$\scriptscriptstyle \pmb \lbrack$}}
\newcommand{\brb}{\raise.3ex\hbox{$\scriptstyle \pmb \rbrack$}}
\newcommand{\sbrb}{\raise.1ex\hbox{$\scriptscriptstyle \pmb \rbrack$}}
\newcommand{\bla}{\raise.2ex\hbox{$\scriptstyle\pmb \langle$}}
\newcommand{\sbla}{\raise.1ex\hbox{$\scriptscriptstyle\pmb \langle$}}
\newcommand{\bra}{\raise.2ex\hbox{$\scriptstyle\pmb \rangle$}}
\newcommand{\sbra}{\raise.1ex\hbox{$\scriptscriptstyle\pmb \rangle$}}
\newcommand{\blrb}{\raise.3ex\hbox{$\scriptstyle \pmb | $}}
\newcommand{\sblrb}{\raise.1ex\hbox{$\scriptscriptstyle \pmb | $}}

\newcommand{\R}{\mathbb R}

\newcommand{\psum}{{+_{\negthinspace\kern-2pt p}}\,}
\newcommand{\qsum}[1]{{+_{\negthinspace\kern-2pt #1}}\,}
\newcommand{\dpsum}{{\tilde+_{\negthinspace\kern-1pt p}}\,}
\newcommand{\dqsum}[1]{{\tilde+_{\negthinspace\kern-1pt #1}}\,}
\newcommand{\lsub}[1]{\hskip -1.5pt\lower.5ex\hbox{$_{#1}$}}

\numberwithin{equation}{section}

\newtheorem{theo}{Theorem}[section]
\newtheorem{coro}[theo]{Corollary}
\newtheorem{lemm}[theo]{Lemma}
\newtheorem{prop}[theo]{Proposition}

\theoremstyle{definition}

\begin{document}

\title[]{log-concavity of eigenfunction and Brunn-Minkowski inequality of eigenvalue for weighted p-Laplace operator}

\author[L. Qin]{Lei Qin }
\address{Institute of Mathematics,
Hunan University, Changsha, 410082, China}
\email{lqin8763@gmail.com}

%\thanks will become a 1st page footnote.
%\thanks{}

\maketitle

\baselineskip18pt

\parskip3pt

\begin{abstract}
In this paper, we investigate the log-concavity property of the first eigenfunction to the weighted $p$-Laplace operator in class of bounded, convex and smooth domain. Moreover, we prove a Brunn-Minkowski-type inequality for the first eigenvalue to the weighted $p$-Laplace operator in the class of $C^2$ convex bodies in $\R^n$.
\end{abstract}

\section{Introduction}

In convex geometry, one of fundamental content is the Brunn-Minkowski inequality. One form of the Brunn-Minkwoski inequality states that if $K$ and $L$ are convex bodies (compact convex sets with nonempty interior) in $\R^n$ and $0<\lambda<1$, then
\[
V((1-\lambda)K+\lambda L)^{1/n} \geq (1-\lambda)V(K)^{1/n}+\lambda V(L)^{1/n},
\]
where $V$ and $+$ denote volume and vector sum. Equality holds precisely when $K$ and $L$ are equal up to translation and dilatation. Furthermore, the inequality also holds for any non-empty bounded measurable sets in $\R^n$. In 1978, Osserman \cite{Osserman1978} stated the classical Brunn-Minkwoski inequality. And Osserman emphasizes that the Brunn-Minkwoski inequality implies the classical isoperimetric inequality. The relationship between Brunn-Minkowski inequality and other inequalities in geometry and analysis, and some application, see Gardner \cite{Gardner2002}. As far as we know, there are many examples of functional satisfying a Brunn-Minkowski type inequality: the $(n-1)$-dimensional measure of boundary, the variational functionals, such as torsional rigidity, Newton Capacity, $p$-Capacity, the first eigenvalue of the Laplacian, see \cite{Borell1983,Borell1985,Brascamp1976,Caffarelli-Jerison1996,Colesanti2005,Andrea-Cuoghi2006,Andrea-Paolo2003} and so on.

The convexity properties of solution to PDEs have been extensively studied. In other words, the classical Brunn-Minkowski inequality is equivalent to the fact that $n$-dimensional volume raised to the power $1/n$ is concave. Specially, many geometric properties of the pseudo-Laplacian $\Delta_p:=-\mathrm{div} (|\nabla \cdot|^{p-2}\nabla \cdot)$ has been studied by many authors. The typical eigenvalue problem for the $p$-Laplace is to find $\lambda\in \R$ and $u\not\equiv 0\in W^{1,p}_0(\Omega)$, which is a weak solution of
\begin{equation}\nonumber
\left\{
\begin{aligned}
& -\text{div}(|\nabla u|^{p-2} \nabla u)=\lambda |u|^{p-2}u, \ \ \ \ \ \  \mathrm {in} \ \Omega,  \\
& u=0,  \ \ \ \ \ \  \ \ \ \ \ \ \ \ \  \ \ \ \ \mathrm {on} \ \partial \Omega,
\end{aligned}
\right.
\end{equation}
where $\Omega$ ia a bounded domain in $\R^n$ and $W^{1,p}_0(\Omega)$ is the Sobolev space $(p\geq 1)$. The existence, uniqueness and regularity of $(\lambda,u)$ have been studied in \cite{Azorero1987,Barles1988,Sakaguchi1987}. Azorero and Alonzo \cite{Azorero1987} proved the existence of an increasing sequence $(\lambda_k)_{k\geq 1}$ of the eigenvalues such that $\lambda_k\rightarrow +\infty$. Sakaguchi \cite{Sakaguchi1987} proved that there exists a unique, positive, weak solution to the first eigenvalue of above equation, where $\Omega$ is a bounded domain in $\R^n$ $(n\geq 2)$ with smooth boundary $\partial \Omega$. Furthermore, if $\Omega$ is convex, then $u$ is a log-concave function. Barles \cite{Barles1988} also gave some remarks on uniqueness results of the first eigenvalue of $p$-Laplace independently.  In this paper, our first result is to investigate an extension to the first eigenfunction of the weighted $p$-Laplace. Specially, for case $p=2$, the corresponding results are studied in \cite{Andrea-Paolo2024} very recently, from the view of viscosity solution and constant rank theorem.

Let $\Omega$ be a convex body of $\R^n$ with smooth boundary $\partial \Omega$, we prove that there exists a unique, positive weak solution $u\in W^{1,p}_0(\Omega,\gamma)$ to equation
\begin{equation}\label{PDE1}
\left\{
\begin{aligned}
& \text{div}(|\nabla u|^{p-2} \nabla u)-(x,\nabla u)|\nabla u|^{p-2}=-\lambda_{p,\gamma} |u|^{p-2}u, \ \  \mathrm {in} \ \Omega,  \\
& u=0,  \ \ \ \  \ \mathrm {on} \quad \partial \Omega, \ \ \ \ u>0\quad \ \ \mathrm {in}\quad \Omega,
\end{aligned}
\right.
\end{equation}
where $W^{1,p}_0(\Omega,\gamma)$ is the Gaussian Sobolev space $(1<p<+\infty$), see Section \ref{Se2}. Multiplying the above equation by $u$ and integrating by parts, one sees that
\[
\lambda_{p,\gamma}(\Omega)=\frac{\int_{\Omega} |\nabla u|^{p} d\gamma(x)}{\int_\Omega |u|^{p} d\gamma(x)}.
\]
Furthermore, the variational definition to the first eigenvalue $\lambda_{\gamma,p}(\Omega)$ is given by
\begin{equation}\label{varia-prob}
\lambda_{p,\gamma}(\Omega)=\inf \left\{\frac{\int_{\Omega} |\nabla u|^{p} d\gamma(x)}{\int_\Omega |u|^{p} d\gamma(x)},\ u\in\ W^{1,p}_0(\Omega,\gamma),\ u \not\equiv 0\right\}.
\end{equation}
By the Sobolev embedding constant, we know that $\lambda_{p,\gamma}>0$ when $|\Omega|>0$.

Under suitable assumptions, we prove that the first eigenfunction $u$ for weighted $p$-Laplace operator is log-concave.

\begin{theo}\label{T2}
Assume that $\Omega $ is a bounded convex domain in $\R^n$ with boundary of class $C^2$. Fixed a number $p>1$. Let $u\in\ W^{1,p}_0(\Omega,\gamma)$ be a positive, weak solution of \eqref{PDE1}, then the function
\[
w=-\ln u
\]
is convex in $\Omega$.
\end{theo}

 The first part of this paper is devoted to the existence, uniqueness and regularity for the weak solution of \eqref{PDE1}. The existence is proved by a classical variational method for generalized quasi-linear equation in divergence form. To obtain uniqueness, we need to prove the corresponding weak comparison theorem and Hopf's Lemma to weighted $p$-Laplace operator, and so on.  For the case $\Omega$ is a ball, it is easy to prove the existence of a radially symmetric eigenfunction by direct calculation, which is smooth except possibly at zero.

As for the global $C^{1,\alpha}$ $(0<\alpha<1)$ regularity to \eqref{PDE1}. First of all, we can get the $C^{1,\alpha}_{\mathrm{loc}}$ estimates by the results of Di Benedetto \cite{Benedetto1983} or Tolksdof\cite{Tolksdorf1984}. Then, we only need to prove the boundary estimates. By adapting a method of Tolksdorf \cite{Tolksdorf1983} which replaces the boundary estimates by an interior estimates via the Schwarz reflection principle, we can deduce the global estimate. Note that the proof consists essentially in obtaining boundary estimates in $C^{1,\alpha}$ by using the method of Tolksdorf \cite{Tolksdorf1983} and the estimates of Di Benedetto \cite{Benedetto1983}, see also \cite[p.69-71]{Barles1988} for more details.

On the other hand, to establish the log-concavity in Theorem \ref{T2}, we need to use the Korevaar's concavity maximum principle, but this method works only for classical solutions. However, in the case of the pseudo-Laplacian, there is generally only weak solutions, since the pseudo-Laplacian is degenerate elliptic. Precisely, the bounded solutions to the pseudo-Laplacian belong to $C^{1+\alpha}(\overline{\Omega})$ for some $\alpha$ ($0<\alpha<1$) and not always belong to $C^2(\Omega)$. Therefore, we can not directly apply concavity maximum principle to our problem, that is, we need to consider a regularized problem, then using a approximation process to come back to our original problem.

In the following, for $n$-dimensional convex bodies, we prove a Brunn-Minkowski type inequality for the first eigenvalue of the weighted $p$-Laplace operator.

\begin{theo}\label{T1}
Let $\Omega_0,\Omega_1$ be convex bodies in $\R^n$ with boundary of class $C^2$, and $p>1$. Let $t\in [0,1]$ and set
\[
\Omega_t=(1-t)\Omega_0+\Omega_1.
\]
Then, the following inequality holds
\[
\lambda_{p,\gamma}(\Omega_t)\leq (1-t)\lambda_{p,\gamma}(\Omega_0)+t\lambda_{p,\gamma}(\Omega_1).
\]
\end{theo}

To establish Theorem \ref{T1}, we shall need the regularity and log-concavity properties, which are obtained in Theorem \ref{T2} or Section \ref{Se3}. Finally, we remark that due to the lack of homogeneity of the eigenvalue $\lambda_{p,\gamma}(\Omega)$, we can not get the weaker version for Brunn-Minkowski type inequality. This is a difference between the classical and Gaussian case.

The organization of the paper is as follows. In Section \ref{Se2}, we list some basic facts and notions in convex geometry and partial differential equation. In Section \ref{Se3}, we prove the existence, uniqueness and regularity to the first eigenfunction to the weighted $p$-Laplace operator, under the smooth assumption. In Section \ref{Se4}, we prove the log-concavity of the first eigenfunction to weighted $p$-Laplace. In Section \ref{Se5}, we are devoted to solving the Brunn-Minkowski type inequality for the first eigenvalue to weighted $p$-Laplace..

\section{preliminaries}\label{Se2}

Let $\mathbb{R}^n$ be the $n$-dimensional Euclidean space, $x$ is a point in $\R^n$, $|x|$ is the Euclidean norm of $x$ and $dx$ denotes integration with respect to Lebesgue measure in $\R^n$. For $x\in \R^n$ and $r>0$, we denote by $B_{r}(x)$ the open ball of radius centered in $x$. If $\Omega\subset \R^n$, we denote the closure, interior, and the boundary of $\Omega$ by $\overline{\Omega}$, $\mathrm{int} \Omega$ and $\partial \Omega$. Let $C_{c}^{\infty}(\Omega)$ be the set of functions from $C^{\infty}(\Omega)$ having compact support in $\Omega$. We denote the support set of a real-value function by $\mathrm{spt}(u)$. Let $\mathcal{K}^n$ be the class of convex bodies in $\mathbb{R}^n$. We call $\Omega\in \mathcal{K}^n$ strongly convex, if all its the principal curvatures of $\partial \Omega$ are positive.

Let $\Omega_1, \Omega_2 \in \mathcal{K}^n $, then for $\lambda \in [0,1]$, the convex linear combination of $\Omega_1$ and $\Omega_2$ is defined by
\[
\lambda \Omega_1 +(1-\lambda)\Omega_2=\{\lambda x+(1-\lambda)y:\ x\in \Omega_1,\ y\in \Omega_2 \}.
\]
For a generic smooth function $u$, by a direct calculation, we can write as follows:
\begin{equation}\label{p-Lapla-form}
\Delta_p u(x)=|\nabla u(x)|^{p-2} \Delta u(x)+(p-2)|\nabla u|^{p-4}\sum_{i,j=1}^n  \frac{\partial^2 u}{\partial x_i \partial x_j}\frac{\partial u}{\partial x_i} \frac{\partial u}{\partial x_j}.
\end{equation}
Moreover, for a point $x$ such that $\nabla u(x)\neq 0$, by a direct calculation, we can write
\begin{equation}\label{p-Lapla-form2}
\Delta_p u(x)=|\nabla u(x)|^{p-2}\left( \Delta u(x)+(p-2)\frac{\partial^2 u(x)}{\partial n^2}\right),
\end{equation}
where $n=\nabla u (x) /|\nabla u(x)|$ and
\[
\frac{\partial^2 u(x)}{\partial n^2}=\langle D^2u(x)n,n\rangle.
\]
Here, $|\nabla u|^{p-2}\nabla u$ is defined to be zero at each $x$ where $\nabla u(x)=0$. Note that if $u$ has continuous second partial derivative in $\Omega$, then $\Delta_p=0$ in weak sense and the divergence theorem imply that $\Delta_pu=0$ in point-wise sense in $\Omega$.

We denote the Gauss probability space by $( \mathbb{R}^{n}, \gamma)$ with the measure
$$
\gamma(\Omega)= \frac{1}{(2\pi)^{n/2}} \int_{\Omega} e^{-|x|^2/2} dx,
$$
for any measurable set $\Omega\subseteq \mathbb{R}^{n}$, $d\gamma$ stands for integration with respect to $\gamma$, and $d\gamma_{\partial \Omega}$ is the $(n-1)$-dimensional hausdorff measure with respect to Gaussian measure $\gamma$.

In this paper, we define the weighted $p$-Laplace operator as follows:
\begin{equation}\label{OU2}
\Delta_{p,\gamma} u:= \mathrm{div} (|\nabla u|^{p-2}\nabla u)-(x,\nabla u)|\nabla u|^{p-2},
\end{equation}
which also satisfies the integration by parts identity for the standard Gaussian measure:
\[
\int_{\Omega} \upsilon \Delta_{p,\gamma} u d\gamma=-\int_\Omega |\nabla u|^{p-2}\langle \nabla u, \nabla \upsilon \rangle d\gamma+\int_{\partial \Omega} \upsilon |\nabla u|^{p-2} \langle \nabla u, n_x \rangle d\gamma_{\partial \Omega},
\]
for functions $u,\upsilon\in W^{1,p}(\Omega,\gamma)$, and a Lipschitz domain $\Omega$.

Based on the integration by parts, we define the weak solution to \eqref{PDE1}, that is, for $\varphi\in C^{\infty}_c(\Omega)$,
\begin{equation}\label{weak-solution}
\int_{\Omega} |\nabla u|^{p-2} \nabla u \cdot \nabla \varphi\ d\gamma (x) = \lambda_{p,\gamma} \int_{\Omega}|u|^{p-2}u\cdot \varphi \ d\gamma (x).
\end{equation}
Indeed, by a compactness, it is obvious that there exists a weak solution for the variational problem \eqref{varia-prob}.

Let $\Omega\subseteq \R^n$ be a measurable set, and $|\Omega|$ be the Lebesgue measure. Given $1<p<+\infty$, let $L^p(\Omega)$ denote the usual Lebesgue space of functions $u$ on $\Omega$ with norm
\[
\| u\|^p_{p}:=\int_{\Omega} |u(x)|^p dx<+\infty.
\]
Let $W^{1,p}(\Omega)$ denote the Sobolev space whose functional elements $u$ and their distributional partial derivatives $u_{x_i}$, $1\leq i\leq n$ belong to $L^{p}(\Omega)$ with norm
\[
\|u \|^p_{1,p}=\|u\|^p_p+\mathop{\sum}_{i=1}^n \| u_{x_i}\|^p_p.
\]
Furthermore, $W^{1,p}_0(\Omega)$ denotes the Sobolev space that the closure of $C^\infty_c(\Omega)$ in $W^{1,p}(\Omega)$. In the following, we define the Sobolev space with Gauss weight. For $p>1$, we denote $L^p(\Omega,\gamma)$ the space of all measurable functions $u$ such that
\begin{equation}\nonumber
L^p(\Omega,\gamma)=\left\{ u:\Omega\rightarrow \R:\ \|u\|^{p}:=\int_{\Omega} |u(x)|^p d\gamma (x)<+\infty \right\},
\end{equation}
and
\begin{equation}\nonumber
W^{1,p}(\Omega,\gamma)=\left\{ u:\Omega\rightarrow \R \ \| u\|^{p}:=\int_{\Omega} |\nabla u(x)|^{p} +|u(x)|^p d\gamma (x)<+\infty \right\}.
\end{equation}
Furthermore, $W^{1,p}_0(\Omega,\gamma)$ denotes the Sobolev space that the closure of $C^\infty_c(\Omega)$ in $W^{1,p}(\Omega,\gamma)$. When $\Omega$ is bounded domain, then function $e^{-|x|^2/2}$ has positive upper and lower bounds in $\Omega$, it is obvious that $W^{1,p}_0(\Omega,\gamma)=W^{1,p}_0(\Omega)$, that is, if $\varphi\in W^{1,p}_0(\Omega,\gamma)$, then $\varphi\in W^{1,p}_0(\Omega)$; Vice versa. However, we only have $W_0^{1,p}(\Omega)\subset W_0^{1,p}(\Omega,\gamma)$ when $\Omega$ is a unbounded domain. (cf. \cite{Tero1994} and its references).

For an open, bounded, Lipschitz domain $\Omega$, we have the following Sobolev embedding. When $p>n$, we have $W^{1,p}_0(\Omega)\hookrightarrow C^{0,\alpha}(\overline{\Omega})$ for some $0<\alpha<1$. For $p=n$, $W^{1,p}_0(\Omega)\hookrightarrow L^{p}(\overline{\Omega})$ for all $p\in [1,+\infty)$. For $1<p<n$, $W^{1,p}_0(\Omega)\hookrightarrow L^{q}(\Omega)$ where $1 \leq q \leq p^{*}$, $p^{*}=np/{(n-p)}$. Moreover, $W^{1,p}_0(\Omega)\hookrightarrow L^p(\Omega)$ is compact for $p\geq 1$, (cf. \cite{Brezis2011,Radulescu2019}).

\section{Existence, uniqueness and regularity of weak solution }\label{Se3}

In this section, we prove the existence, positivity, uniqueness and regularity of weak solution to equation \eqref{PDE1}.
\begin{theo}\label{weak-solution2}
Given $p>1$, let $\Omega$ be a bounded domain in $\R^n$ $(n\geq 2)$ with boundary of class $C^2$. Then there exists a non-trivial, non-negative weak solution $u\in W^{1,p}_0(\Omega,\gamma)$ to the eigenvalue problem \eqref{PDE1}. Moreover, any non-trivial solution is positive or negative in $\Omega$ and belongs to $C^{1+\alpha}(\overline{\Omega})$ for some $\alpha$ $(0<\alpha<1)$. Furthermore, if the boundary $\partial \Omega$ is connected, the solution is unique up to a multiplicative constant.
\end{theo}

\begin{prop}\label{Transform}
Given $p>1$, let $\Omega$ be a bounded domain in $\R^n$ $(n\geq 2)$ with boundary of class $C^2$. If $u\in W_0^{1,p}(\Omega,\gamma)$ is a weak solution of \eqref{PDE1}, then $u\in W_0^{1,p}(\Omega)$ is a weak solution of the following equation
\begin{equation}\label{PDE3}
\left\{
\begin{aligned}
&\mathrm{div} (e^{-|x|^2/2}|\nabla u|^{p-2} \nabla u) =-\lambda_{p,\gamma}e^{-|x|^2/2} |u|^{p-2}u , \ \ \  \mathrm {in}\ \ \Omega,  \\
& u=0,  \ \ \ \ \ \  \ \ \ \ \ \ \ \ \  \ \ \ \ \ \ \ \ \ \mathrm {on} \ \partial \Omega.
\end{aligned}
\right.
\end{equation}
And, vice versa, every solution of \eqref{PDE3} is a weak solution of \eqref{PDE1}.
\end{prop}

\begin{proof}[\bf Proof] For any $\psi\in W_0^{1,p}(\Omega) $, we need to prove that
\[
\int_{\Omega} e^{-|x|^2/2} |\nabla u|^{p-1} \nabla  u \cdot \nabla \psi dx =\lambda_{p,\gamma} \int_{\Omega} e^{-|x|^2/2} |u|^{p-2}u \cdot \psi dx.
\]
Since $W^{1,p}_0(\Omega,\gamma)=W^{1,p}_0(\Omega)$, we have $\psi\in W_0^{1,p}(\Omega,\gamma)$. Since $u$ is a weak solution of \eqref{PDE1}, we have
\begin{equation}
\int_{\Omega} |\nabla u|^{p-1} \nabla u  \cdot \nabla\psi d\gamma (x)=\lambda_{p,\gamma} \int_{\Omega} |u|^{p-2}u e^{-|x|^2/2} \cdot \nabla\psi d\gamma (x).
 \end{equation}
This completes its proof. The vice versa is analogous.
\end{proof}

In the following, we prove the weak comparison principle for the weighted $p$-Laplace operator.
\begin{prop}\label{weak-comparison-princ}
Given $p>1$, let $\Omega$ be a bounded domain in $\R^n$ $(n\geq 2)$ with boundary of class $C^2$. Let $u_1,u_2\in W^{1,p}(\Omega,\gamma)$ satisfy
\begin{equation}
\int_{\Omega} |\nabla u_1|^{p-1} \nabla u_1\nabla \varphi d\gamma(x) \leq \int_{\Omega} |\nabla u_2|^{p-1} \nabla u_2 \cdot \nabla \varphi d\gamma(x)
\end{equation}
for all non-negative $\varphi\in W_{0}^{1,p}(\Omega,\gamma)$, that is
\[
-\Delta_p u_1+(x,\nabla u_1)|\nabla u_1|^{p-2} \leq -\Delta_p u_2+(x,\nabla u_2)|\nabla u_2|^{p-2} \quad \mathrm{in}\ \Omega
\]
in the weak sense. Then the inequality
\[
u_1\leq u_2,\ \ \ \mathrm{on}\ \ \partial \Omega
\]
implies that
\[
u_1 \leq u_2,\ \ \ \mathrm{in}\ \ \Omega.
\]
\end{prop}

\begin{proof}
Since $W^{1,p}(\Omega,\gamma)=W^{1,p}(\Omega)$, we have $u_1,u_2\in W^{1,p}(\Omega)$, and for any non-negative function $\varphi\in W_0^{1,p}(\Omega)$, it holds
\[
\int_{\Omega}  e^{-|x|^2/2} |\nabla u_1|^{p-2}\nabla u_1 \cdot \nabla \varphi dx \leq \int_{\Omega}  e^{-|x|^2/2}  |\nabla u_1|^{p-2}\nabla u_1 \cdot \nabla \varphi dx,
\]
that is,
\[
-\mathrm{div}\ (e^{-|x|^2/2}|\nabla u_1|^{p-1}\nabla u_1)\leq -\mathrm{div}\ (e^{-|x|^2/2}|\nabla u_1|^{p-1}\nabla u_1)\ \ \mathrm{in}\ \Omega,
\]
in the weak sense w.r.t. Sobolev space $W^{1,p}_0(\Omega)$. Let $\psi=\max \{ u_1-u_2,0\}$. Since $u_1\leq u_2$ on $\partial \Omega$, so $\psi$ belongs to $W^{1,p}_0(\Omega)$. Inserting this function $\psi$ into above equation as a test function, we have
\[
\int_{\{ u_1>u_2\}} (|\nabla u_1|^{p-2}\nabla u_1-|\nabla u_2|^{p-2}\nabla u_2) \cdot (\nabla u_1-\nabla u_2)d\gamma(x)\leq 0.
\]
We set
\[
a_j(x,u,\nabla u)=e^{-|x|^2/2}|\nabla u|^{p-2} u_j,
\]
and
\[
a(x,u,\nabla u)=e^{-|x|^2/2}\lambda_{p,\gamma} |u|^{p-2}u .
\]
By \cite[Lemma 1, p.129]{Tolksdorf1984}, we obtain
\[
u_1\leq u_2\ \ \mathrm{in}\ \Omega.
\]
\end{proof}

To prove Hopf's Lemma for the weighted $p$-Laplace operator, we need the following Proposition.
\begin{prop}\label{P1-Cap}
Given $p>1$ and $r_1>r_0>0$. The function $u: \overline{B}_{r_1}(0)\setminus {B}_{r_0}(0)\rightarrow \R$
\begin{equation}
u(x):=u(|x|)=\varphi(r)=\frac{\int_{r_0}^{r}[ e^{{t^2}/{2}}t^{-n+1} ]^{p-1}dt}{\int_{r_0}^{R}[ e^{{t^2}/{2}}t^{-n+1} ]^{p-1}dt}
\end{equation}
is the unique, radially symmetric solution of the problem
\begin{equation}\label{capacity BVP}
\left\{
\begin{array}{lll}
-\mathrm{div} (e^{-|x|^2/2}|\nabla u|^{p-2} \nabla u) =0\ \ \ \mathrm{in}\ \ B_{r_1}(0)\setminus \overline{B}_{r_0}(0),\\
\\
u=0\quad\mathrm{on}\quad \partial B_{r_0}(0),\ \ u=1\quad\mathrm{on}\ \ \partial B_{r_1}(0),\\
\\
0<u<1\quad \mathrm{in}\quad  B_{r_1}(0) \setminus B_{r_0}(0), \quad \mathrm{and}\quad |\nabla u|\geq c> 0\\
 \\
\mathrm{for\ some\ positive\ constant}\ c.
\end{array}
\right.
\end{equation}
\end{prop}

\begin{proof}[\bf Proof]
The function $u$ is a radially symmetric, that is, $u(x)=\varphi(r)$ whenever $r=|x|$. Since
\begin{equation}\nonumber
\frac{\partial r}{\partial x_i}=\frac{x_i}{r},\ \ \ \ \frac{\partial u}{\partial x_i}=\frac{\partial \varphi}{\partial r} \frac{\partial r}{\partial x_i}=\varphi'(r)\frac{x_i}{r},
\end{equation}
and
\[
\frac{\partial^2 u}{\partial x_i \partial x_j}=\varphi''(r)\frac{x_i x_j}{r^2}+\varphi'(r)\frac{r^2\delta_{ij}-x_i x_j}{r^3}.
\]
Then, we have
\[
\nabla u=\varphi'(r)\frac{x}{r},\ \ \ (x,\nabla u)=\varphi'(r)r.
\]
Moreover, we observe that for a generic smooth function $u$ and a point $x$ such that $\nabla u(x)\neq 0$, by a direct calculation, we can write
\begin{equation}\nonumber
\Delta_p u(x)=|\nabla u(x)|^{p-2} \Delta u(x)+(p-2)|\nabla u|^{p-4}\sum_{i,j=1}^n  \frac{\partial^2 u}{\partial x_i \partial x_j}\frac{\partial u}{\partial x_i} \frac{\partial u}{\partial x_i}  .
\end{equation}
Passing to polar coordinates, and the above facts, we know that \eqref{capacity BVP} hold. Therefore, by a comparison principles for divergence form operator, for $r_0\leq r\leq r_2$,
\[
u(x)=u(|x|)=\varphi(r)=\frac{\int_{r_0}^{r}[ e^{{t^2}/{2}}t^{-n+1} ]^{p-1}dt}{\int_{r_0}^{R}[ e^{{t^2}/{2}}t^{-n+1} ]^{p-1}dt}
\]
is the unique solution of \eqref{capacity BVP}. This completes the proof.
\end{proof}

\begin{prop}\label{Hopf-Lemma}
Let $\Omega$ be a bounded domain in $\R^n$ $(n\geq 2)$ with boundary of class $C^2$. Let $u\in C^1(\overline{\Omega})$ satisfy
\begin{equation}\label{PDE4}
\left\{
\begin{aligned}
&\mathrm{div} (e^{-|x|^2/2}|\nabla u|^{p-2} \nabla u)\geq 0, \ \ \  \mathrm {in}\ \ \Omega,  \\
& u>0,  \ \ \ \mathrm {in} \ \Omega, \ \mathrm{and}\ \ u=0,\ \ \ \mathrm {on} \ \partial \Omega,
\end{aligned}
\right.
\end{equation}
in the weak sense w.r.t Sobolev space $W_0^{1,p}(\Omega)$. Then $\frac{\partial u}{\partial \nu}<0$ on $\partial \Omega$, where $\nu$ denotes the unit exterior normal vector to $\partial \Omega$.
\end{prop}

\begin{proof}[\bf Proof]
Let $x_0\in \partial \Omega$, there exists an open ball $B_{2r}(y)\subset \Omega$ with $x_0\in \partial B_{2r}(y)\cap \partial \Omega$, where $B_{2r}(y)$ denotes an open ball in $\R^n$ centered at $y$ with radius $2r$. By Proposition \ref{P1-Cap}, we can find a smooth solution $\upsilon$ of
\begin{equation}
\left\{
\begin{array}{lll}
-\mathrm{div} (e^{-|x-y|^2/2}|\nabla \upsilon|^{p-2} \nabla \upsilon) =0\ \ \ \mathrm{in}\ \ B_{2r}(y)\setminus \overline{B_{r}(y)},\\
\\
\upsilon=0\quad\mathrm{on}\quad \partial B_{r}(y),\ \ u=1\quad\mathrm{on}\ \ \partial B_{2r}(y),\\
\\
0<\upsilon<1\quad \mathrm{in}\quad  B_{2r}(y) \setminus B_{r}(y), \quad \mathrm{and}\quad |\nabla \upsilon|\geq c> 0,\\
 \\
\mathrm{for\ some\ positive\ constant}\ c.
\end{array}
\right.
\end{equation}
Since $u$ is positive in $\Omega$, we have
\[
\tau=\inf \{ u(x):\ x\in \partial B_r(y)\}>0,
\]
Set $w=\tau \upsilon$, then $w$ satisfies
\begin{equation}
\left\{
\begin{array}{lll}
-\mathrm{div} (e^{-|x-y|^2/2}|\nabla w|^{p-2} \nabla w) =0\ \ \ \mathrm{in}\ \ B_{2r}(y)\setminus \overline{B_{r}(y)},\\
\\
w=\tau\quad\mathrm{on}\quad \partial B_{r}(y),\ \ w=0\quad\mathrm{on}\ \ \partial B_{2r}(y).
\end{array}
\right.
\end{equation}
Since $w\leq u$ on $\partial (B_{2r}(y)\setminus \overline{B_{r}(y)})$, we can apply \eqref{weak-comparison-princ} to $w$ and $u$, with $u_1=w$ and $u_2=u$, to obtain
\[
w\leq u\ \mathrm{in}\ \  B_{2r}(y)\setminus \overline{B_{r}(y)}.
\]
Since $w(x_0)\leq u(x_0)$, we deduce that
\[
\frac{\partial u}{\partial \nu}\leq \frac{\partial w}{\partial \nu}=\tau \frac{\partial \upsilon}{\partial \nu}<0 \ \mathrm{at}\ x_0\in \partial \Omega.
\]
This completes its proof.
\end{proof}

\begin{proof}[\bf Proof of Theorem \ref{weak-solution2}]
By Proposition \ref{Transform}, the existence of weak solution of \eqref{PDE1} in Sobolev space $W^{1,p}_0(\Omega,\gamma)$ is equivalent to existence of the solution of quasi-linear elliptic equation \eqref{PDE3} in divergence form w.r.t Sobolev space $W_0^{1,p}(\Omega)$. By a standard variational method (cf. \cite[Chapter 5]{Lady-Ural1968}), we can deduce that there exists at least one solution $u$ to this variational problem. Since $|u|\in W_0^{1,p}(\Omega,\gamma)$ is also a solution, we may assume that $u\geq 0$.

By the Sobolev embedding theorem, when $p>n$, we know that $W^{1,p}_0(\Omega)\hookrightarrow C^{0,\alpha}(\overline{\Omega})$ for some $0<\alpha<1$. For $p=n$, $W^{1,p}_0(\Omega)\hookrightarrow L^{p}(\overline{\Omega})$ for all $p\in [1,+\infty)$. we only need to prove the case that if $1<p<n$, we have $u\in L^{\infty}(\Omega)$. We define $\upsilon$ by
\[
\upsilon=\inf \{ u,M\},\quad M\geq 0,
\]
and we take $\varphi=\upsilon^{kp+1}$ $(k\geq 0)$ in \eqref{weak-solution}, we obtain
\[
(kp+1)\int_{\Omega} | \upsilon|^{kp} \cdot |\nabla \upsilon|^{p}\ d\gamma (x) = \lambda_{p,\gamma} \int_{\Omega}u^{p-1}\cdot \upsilon^{kp+1} \ d\gamma (x).
\]
The left-hand side is equal to
\[
\frac{kp+1}{(k+1)^p} \int_{\Omega} |\nabla (\upsilon^{k+1})|^p d\gamma(x).
\]
As $W^{1,p}_0(\Omega)\hookrightarrow L^q(\Omega)$ with $q=\frac{np}{n-p}$,
\[
\left( \int_\Omega \upsilon^{(k+1)q}d\gamma(x) \right)^{p/q} \leq \frac{c(k+1)^p}{kp+1}\cdot \lambda_{p,\gamma}\int_{\Omega}|u|^{p-1}\cdot \upsilon^{kp+1} \ d\gamma (x).
\]
Finally, if $u\in L^{r}(\Omega)$, we choose $k$ such that
\[
(k+1)p=r,
\]
and the right hand side is estimated by
\[
\frac{c(k+1)^p}{kp+1}\cdot \lambda_{p,\gamma} \int_\Omega |u|^r d\gamma(x).
\]
We conclude that $u\in L^{r'}(\Omega)$, where $r'=(k+1)p=p^{-1} qr$. But $q>p$, so an easy bootstrop argument gives that $u\in L^{\infty}(\Omega)$ (cf. \cite[Section 5 in Chapter 2]{Lady-Ural1968}).

Next, using \cite[Theorem 1.1, p.251]{Lady-Ural1968}, we get $u\in C^{0,\beta}(\overline{\Omega})$ for some $\beta$ $(0<\beta<1)$. Then, using \cite[Theorem 1]{Tolksdorf1984}, we can get $u\in C^{1,\alpha}_{\mathrm{loc}}(\Omega)$ for some $\alpha$ $(0<\alpha<1)$. Similar to the method in \cite[p.69-71]{Barles1988}. we have that $u\in C^{1,\alpha}(\overline{\Omega})$. Furthermore, positivity of $|u|$ follows from Harnack's inequality due to \cite[Theorem 1.1, p.724]{Trudinger1967}.

In the next part, we prove that the weak solution of \eqref{PDE1} is unique up to a multiplicative constant. Let $u_1$ and $u_2\in W_0^{1,p}(\Omega,\gamma) \cap C^{1+\alpha}(\overline{\Omega})$ be two solutions to \eqref{PDE1}. We define the number $b$ by
\begin{equation}\label{uniq-form1}
b=\sup \left\{ \mu\in \R: u_1-\mu u_2>0\quad \mathrm{in}\ \Omega\right\}.
\end{equation}
Applying Proposition \ref{Hopf-Lemma} to $u_1$ and $u_2$, we get
\begin{equation}\label{uniq-form2}
\frac{\partial u_1}{\partial \nu}<0 \quad \mathrm{and}\quad  \frac{\partial u_2}{\partial \nu}<0\quad  \mathrm{on}\ \partial \Omega,
\end{equation}
where $\nu$ denotes the unit exterior normal vector to $\partial \Omega$. Therefore, we deduce that $b$ is positive. Obviously, $u_1-bu_2\geq 0$ in $\Omega$. Furthermore, we can prove that there exists a point $z\in \Omega$ where $u_1-bu_2$ vanishes. We prove this claim by contradiction. Assuming that $u_1-bu_2>0$ in $\Omega$. Since $bu_2$ is also a positive solution to \eqref{PDE1} and $\lambda_{\gamma,p} u_1^{p-1}\geq \lambda_{\gamma,p} (bu_2)^{p-1}$ in $\Omega$, we see that
\begin{equation}\label{uniq-form3}
-\mathrm{div}(e^{-|x|^2/2}\nabla u_1|^{p-2}\nabla u_1)\geq -\mathrm{div}(e^{-|x|^2/2}|\nabla (bu_2)|^{p-2}\nabla (bu_2)),
\end{equation}
in the weak sense w.r.t Sobolev space $W^{1,p}_0(\Omega)$.

Since $\partial \Omega$ is smooth, there exists a smooth vector field $\nu(x)$ on a neighborhood of $\partial \Omega$ in $\R^n$, which is equal to the unit exterior normal vector to $\partial \Omega$ for all $x\in \partial \Omega$. By Proposition \ref{Hopf-Lemma}, the continuity of derivatives $\frac{\partial u_1}{\partial \nu}$ and $\frac{\partial u_2}{\partial \nu}$, and \eqref{uniq-form2}, we have
\begin{equation}\label{uniq-form4}
\frac{\partial u_1}{\partial \nu}<-\delta \quad \mathrm{and}\quad  \frac{\partial u_2}{\partial \nu}<-\delta\quad  \mathrm{on}\ \overline \Gamma,
\end{equation}
where $\delta$ is a positive constant and $\overline{\Gamma}$ is an open connected neighborhood of $\partial \Omega$ in $\Omega$. Since $\partial \Omega$ is connected, we can choose $\Gamma$ to be connected. It is obvious that $|\nabla u_1|\geq \delta$ and $|\nabla u_2|\geq \delta$ in $\Gamma$. Theorefore, it follows from the regularity theory of the elliptic partial differential equations, see \cite{Book-Gilbarg-Trudinger} that $u_1$ and $u_2$ belong to $C^{\infty}(\Gamma)$. Moreover, by \eqref{uniq-form4}, we have
\begin{equation}\label{uniq-form5}
t\frac{\partial u_1}{\partial \nu}+(1-t)\frac{\partial (bu_2)}{\partial \nu}\leq  -\min \{ \delta,b\delta\} <0,\quad \mathrm{in}\ \Gamma,
\end{equation}
for all number $t$ $(0<t<1)$. By formula \eqref{uniq-form3} and mean value theorem, we have
\begin{align*}
0\leq -&\mathrm{div}(e^{-|x|^2/2}|\nabla u_1|^{p-2}\nabla u_1)- \{ -\mathrm{div}(e^{-|x|^2/2}|\nabla (bu_2)|^{p-2}\nabla (bu_2))\}\\
&=-\mathrm{div}(e^{-|x|^2/2}\nabla u_1|^{p-2}\nabla u_1- e^{-|x|^2/2}|\nabla (bu_2)|^{p-2}\nabla (bu_2))\\
&=-\sum^{n}_{i,j=1}\frac{\partial}{\partial x_i} [a^{i,j}(x)\frac{\partial}{\partial x_j}(u_1-bu_2)]\quad \mathrm{in}\ \Gamma,
\end{align*}
where
\[
a^{i,j}(x)=e^{-|x|^2/2} \int^1_0 \frac{\partial a^i}{\partial q_j}[t\nabla u_1+(1-t)\nabla(bu_2)]dt,
\]
and $a^{i}(q)=|q|^{p-2}q_i$ $(i,j=1,\cdots,n)$ for $q=(q_1.\cdots,q_n)\in \R^n$. Define operator
\[
L:=\sum_{i,j}\frac{\partial}{\partial x_i} \left[a^{i,j}(x)\frac{\partial}{\partial x_j}\right].
\]
By \eqref{uniq-form5}, we know that the operator $L$ is uniformly elliptic operator on $\Gamma$. Consequently, we deduce
\begin{equation}\label{uniq-form6}
\left\{
\begin{array}{lll}
-L(u_1-bu_2) \geq 0\ \ \ \mathrm{in}\ \Gamma,\\
\\
u_1-bu_2> 0\quad\mathrm{in}\ \Gamma,\ \ u_1-bu_2=0\ \mathrm{on}\ \partial \Omega.
\end{array}
\right.
\end{equation}
Thus, by Hopf's boundary point Lemma for uniformly elliptic operators (see \cite[Lemma 3.4, p.34]{Book-Gilbarg-Trudinger}), we get
\[
\frac{\partial}{\partial \nu} (u_1-bu_2)<0\quad \mathrm{on}\ \partial \Omega,
\]
where $\nu$ is the exterior normal vector. Therefore, in view of the continuity of functions $u_1,u_2$ and $u_1-bu_2>0$ in $\Omega$, we have
\[
u_1-(b+\eta)u_2\quad \mathrm{in}\ \Omega,
\]
for a positive number $\eta$. This contradicts the definition of number $b$ in \eqref{uniq-form1}. Therefore, we conclude that there exists a point $z\in \Omega$ where $u_1-bu_2$ vanishes.

In the following, we prove that there exists a point $z^\star\in \Gamma$ where $u_1-u_2$ vanishes, and $\Gamma$ is the open connected neighborhood of $\partial \Omega$ in $\Omega$ in \eqref{uniq-form4}. We choose a bounded subdomain $\Omega^\star$ of $\Omega$ with smooth boundary $\partial \Omega^{\star} $ which satisfies
\[
\overline{\Omega}^{\star} \subset \Omega,\quad \partial \Omega^{\star} \subset \Gamma\quad \mathrm{and}\quad z\in \Omega^\star.
\]
Indeed, if we assume that $u_1-bu_2>0$ on $\partial \Omega^\star$, by the continuous, we have
\[
u_1-bu_2\geq \tau_0>0\quad \mathrm{on}\
 \partial \Omega^\star,
\]
for some $\tau_0>0$. We define the function $w$ by $w=bu_2+\tau_0$, by the properties of $u_2$, we deduce
\[
-\mathrm{div}(e^{-|x|^2}|\nabla w|^{p-2}\nabla w)=\lambda_{p,\gamma} e^{-|x|^2}(bu_2)^{p-1} \quad \mathrm{in}\ \Omega^\star,
\]
in the weak sense w.r.t Sobolev space $W^{1,p}_0(\Omega)$, then we have
\begin{equation}\nonumber
\left\{
\begin{array}{lll}
-\mathrm{div}(e^{-|x|^2/2}\nabla u_1|^{p-2}\nabla u_1)\geq  -\mathrm{div}(e^{-|x|^2/2}|\nabla w|^{p-2}\nabla w)\ \ \ \mathrm{in}\ \Omega^\star,\\
\\
u_1\geq w\quad\mathrm{on}\ \partial \Omega^\star.
\end{array}
\right.
\end{equation}
By Proposition \ref{weak-comparison-princ}, we have
\[
u_1\geq w\quad \mathrm{in}\ \Omega^{*}.
\]
Since $z\in \Omega^{*}$, we have $u_1(z)\geq w(z)=bu_2+\tau_0$. This contradicts $u_1(z)-bu_2(z)=0$. Therefore, we conclude that there exists a point $z^\star\in \partial \Omega^{*}\subset \Gamma$ such that $u_1-bu_2$ vanishes.

Applying the strong maximum principle for uniformly elliptic operators (see \cite[Thorem 3.5, p.35]{Book-Gilbarg-Trudinger}), using $u_1-bu_2$ vanishes at $z^\star$, we have
\begin{equation}\label{uniq-form7}
u_1-bu_2=0\quad \mathrm{in}\ \Gamma.
\end{equation}
We also define the number $b^\star$ by
\[
b^\star=\sup \{ \mu\in \R^n: \ u_2-\mu b_1\ \mathrm{in}\ \Omega\}.
\]
By the same argument as above, we have
\begin{equation}\label{uniq-form8}
u_2-b^\star u_1=0\quad \mathrm{in}\ \Gamma.
\end{equation}
Since $u_1$ is positive on $\Gamma$, combining \eqref{uniq-form7} and \eqref{uniq-form8}, we have $bb^\star=1$. Observing that $u_1-bu_2\geq 0$ and $u_2-(1/b)u_1 \geq 0$ in $\Omega$, we obtain $u_1=bu_2$ in $\Omega$. The proof is completed.
\end{proof}

\section{LOG-CONCAVITY FOR THE FIRST EIGENFUNCTION}\label{Se4}

In this section, we always assume that $\Omega$ is a strongly convex (that is, all the principal curvatures of $\partial \Omega$ are positive). For $\delta>0$, we define
\[
\Omega_\delta= \left\{ x\in \Omega:\ \mathrm{dist}(x,\partial \Omega)> \delta \right\}.
\]

\subsection{some properties of the unique, positive solution to (\ref{PDE1})}

\begin{prop}\label{prop1}
There exists a neighborhood $N$ of $\partial \Omega$ in $\Omega$ satisfying the following conditions:
\begin{equation}\nonumber
|\nabla u|\geq \eta>0 \ \ in \ \ \overline{N}\ \ for\ some\ constant\ \eta>0,
\end{equation}
and
\begin{equation}\nonumber
u\in C^2(\overline{N})\cap C^{1+\alpha}(\overline{\Omega})\ for\ some\ \alpha\ (0<\alpha<1).
\end{equation}
\end{prop}
\begin{proof}[\bf Proof]
By Proposition \ref{Hopf-Lemma}, we have
\[
\left|\frac{\partial u}{\partial \nu } \right|>0, \quad \mathrm{on}\quad \partial \Omega,
\]
where $\nu$ denotes the unit exterior normal vector to $\partial \Omega$. Therfore, there exists a neighborhood $N$ of $\partial \Omega$ and a constant $\eta>0$ such that
\begin{equation}\nonumber
|\nabla u|\geq \eta>0 \quad \mathrm{in}\quad \overline{N}.
\end{equation}
Moreover, the equation \eqref{PDE1} becomes strictly elliptic equation in $\overline{N}$, by classical results in \cite{Book-Gilbarg-Trudinger}, we deduce that $u\in C^2(\overline{N})\cap C^{1+\alpha}(\overline{\Omega})$, for some constant $0<\alpha<1$. This completes the proof.
\end{proof}

\begin{prop}\label{prop2}
Let $\Omega$ be strongly convex. For $\delta>0$, we define
\[
\Omega_\delta= \left\{ x\in \Omega:\ \mathrm{dist}(x,\partial \Omega)> \delta \right\}.
\]
There exists a number $\delta_0>0$, the function $\upsilon=\ln u$ satisfies the following:
\begin{center}
The matrix $[-D_{ij} \upsilon]$ is positive on $\Omega \setminus \Omega_{\delta_0}$ where $D_{ij}=\frac{\partial^2}{\partial x_i \partial x_j}$,
\end{center}
and
\begin{flushleft}
For any\ $\delta\ (0<\delta<\delta_0),$ every tangent plane to\ the\ graph\ of\ $\upsilon$ on $\partial \Omega_\delta$ lies above the graph on $\overline{\Omega}_\delta$ and contacts it only at tangent point in $\overline{\Omega}_\delta$.
\end{flushleft}
\end{prop}

\begin{proof}[\bf Proof]
By Proposition \ref{Transform} and Proposition \ref{Hopf-Lemma}, we have
\[
\frac{\partial{u}}{\partial \nu}<0,
\]
where $\nu$ is the exterior normal to $\partial \Omega$. We choose $\delta_0>0$ sufficient small to get $\Omega \setminus \Omega_{\delta_0} \subseteq N$, where $N$ is obtained in Proposition \ref{prop1}. By \cite[Lemma 2.4, p.610-611]{Korevaar1}, this proof is completed.
\end{proof}

\subsection{A regularized problem}
\

Find a solution $u_\varepsilon\in W_0^{1,p}(\Omega,\gamma)$ satisfying
\begin{equation}\label{varia-prob2}
\int_{\Omega} \left( \varepsilon u^2+|\nabla u|^2 \right)^{\frac{p}{2}} d\gamma(x) = \mathop{\min}_{u\in W^{1,p}_0(\Omega,\gamma)} \int_{\Omega} \left( \varepsilon u^2+|\nabla u|^2 \right)^{\frac{p}{2}} d\gamma(x)=\lambda^{\varepsilon}_{p,\gamma}\int_{\Omega} |u|^p d\gamma(x).
\end{equation}
for sufficiently small number $\varepsilon >0$. Without loss of generality, we always assume that $\| u\|_{L^p(\Omega,\gamma)}=1$, for above variational problem.

We set
\[
a_j(x,u,\nabla u)=e^{-|x|^2/2} (\varepsilon u^2+|\nabla u|^2)^{\frac{p-2}{2}}u_j,
\]
and
\[
a(x,u,\nabla u)=\lambda^{\varepsilon}_{p,\gamma} |u|^{p-2}{u}e^{-|x|^2/2}-\varepsilon (\varepsilon u^2+|\nabla u|^2)^{\frac{p-2}{2}} ue^{-|x|^2/2}.
\]
Similar to Proposition \ref{Transform}, we conclude that $u_\varepsilon\in W_0^{1,p}(\Omega)$ is also a weak solution to the equation
\begin{equation}\label{PDE5}
\sum D_j a_j(x,u,\nabla u)+a(x,u,\nabla u)=0, \ \ \mathrm{in}\ \ \Omega.
\end{equation}
And vice versa, every weak solution of \eqref{PDE5} is solution of \eqref{varia-prob2}

\begin{prop}\label{weak-solution1}
There exists at least one non-negative solution $u_{\varepsilon} \in  W^{1,p}_0(\Omega,\gamma) $ to variational problem, which satisfies
\begin{align}\label{regul-prob}
  -\mathrm{div} &\left[ (\varepsilon u^2_{\varepsilon}+|\nabla u_\varepsilon|^2)^{\frac{p-2}{2}}\nabla u_\varepsilon\right] +(x,\nabla u_\varepsilon)(\varepsilon u^2_{\varepsilon}+|\nabla u_\varepsilon|^2)^{\frac{p-2}{2}}  \\
   & =\lambda^{\varepsilon}_{p,\gamma} |u_\varepsilon|^{p-2}{u_\varepsilon}-\varepsilon (\varepsilon u^2_{\varepsilon}+|\nabla u_\varepsilon|^2)^{\frac{p-2}{2}} u_\varepsilon,\nonumber
\end{align}
in $\Omega$, and
\begin{equation}\label{Log-form4}
u_\varepsilon \rightarrow u \ \ \mathrm{in}\ \ W^{1,p}_0(\Omega,\gamma)\ \ \mathrm{as}\ \ \varepsilon \rightarrow 0,
\end{equation}
where $u$ is the unique positive solution to (\ref{PDE1}).
\end{prop}

\begin{proof}[\bf Proof]
By a classical variational method for quasi-linear elliptic equation (cf. \cite[Chapter 5]{Lady-Ural1968}), there exists at least one solution $u_\varepsilon\in W_0^{1,p}(\Omega)$ for quasi-linear elliptic equation \eqref{PDE5}, which is also a solution for variational problem \eqref{varia-prob2} w.r.t Sobolev space $W_0^{1,p}(\Omega,\gamma)$. Since $|u_\varepsilon|\in W_0^{1,p}(\Omega,\gamma)$ is also a solution to $\eqref{varia-prob2}$, we may assume that $u_\varepsilon$ is non-negative in $\Omega$.

Since $\|u_\varepsilon\|_{L^p(\Omega,\gamma)}=1$, from the definition of the variational problem \eqref{varia-prob2}, we have that $\lambda_{p,\gamma}^\varepsilon$ is monotonically increasing w.r.t $\varepsilon>0$. Then, there exists a constant $\varepsilon_0$, for $0<\varepsilon<\varepsilon_0$, we know that the set $\{ u_\varepsilon\}$ is bounded in $W^{1,p}_0(\Omega,\gamma)$. Therefore, there exists a subsequence $\{u_{\varepsilon'}\}$ and $\tilde{u}$ satisfying
\[
u_{\varepsilon'}\rightharpoonup \tilde{u}\quad \mathrm{weakly\ in}\ W^{1,p}_{0}(\Omega,\gamma)\ \mathrm{as}\ \varepsilon'\rightarrow 0.
\]
By the lower semi-continuity of the norm, we have
\begin{equation}\label{Log-form1}
\int_{\Omega} |\nabla \tilde{u}|^p d\gamma(x)\leq \mathop{\lim \inf}_{\varepsilon'\rightarrow 0} |\nabla u_{\varepsilon'}|^pd\gamma(x).
\end{equation}
Since the embedding $W^{1,p}_0(\Omega,\gamma)\hookrightarrow L^p(\Omega,\gamma)$ is compact, we have $u_{\varepsilon'}\rightarrow \tilde{u}$ in $L^{p}(\Omega,\gamma)$ and almost everywhere in $\Omega$ as $\varepsilon'\rightarrow 0$, by taking a subsequence. Then, we have that $\| \tilde{u}\|_{L^p(\Omega,\gamma)}=1$ and $\tilde{u}$ is non-negative in $\Omega$.

On the other hand, we assume that $u$ is the solution of \eqref{PDE1} with $\| u\|_{L^p(\Omega,\gamma)}=1$, by using the minimizing properties of $u$ and $u_\varepsilon$, we deduce that
\begin{equation}\nonumber
\int_\Omega |\nabla u|^p d\gamma(x)\leq \int_\Omega |\nabla u_\varepsilon|^p d\gamma(x)\leq \int_\Omega (\varepsilon u^2_\varepsilon+|\nabla u_\varepsilon|^2)^{\frac{p}{2}}d\gamma(x)\leq \int_{\Omega} (\varepsilon u^2+|\nabla u|^2)^{\frac{p}{2}}d\gamma(x).
\end{equation}
By Lebesgue's dominated convergence theorem, we have
\begin{equation}\label{Log-form2}
\int_\Omega |\nabla u_\varepsilon|^pd\gamma(x)\rightarrow \int_\Omega |\nabla u|^pd\gamma(x)\quad \mathrm{as}\ \varepsilon \rightarrow 0.
\end{equation}
Then, by \eqref{Log-form1} and \eqref{Log-form2}, we deduce that
\[
\int_\Omega |\nabla \tilde{u}|^p d\gamma(x) \leq \int_\Omega |\nabla u|^p d\gamma(x).
\]
Therefore, by the uniqueness of the non-negative solution to \eqref{PDE1}, we deduce that $\tilde{u}=u$. Moreover, we have
\begin{equation}\label{Log-form3}
u_{\varepsilon}\rightharpoonup \tilde{u}\quad \mathrm{weakly\ in}\ W^{1,p}_{0}(\Omega,\gamma)\ \mathrm{as}\ \varepsilon\rightarrow 0.
\end{equation}
Finally, by \eqref{Log-form2} and \eqref{Log-form3}, according to the mean convergence theorem of Riesz and Nagy \cite{Ries-Nagy1955}, We obtain \eqref{Log-form4}. This completes the proof.
\end{proof}

\begin{prop}\label{holder-Regularity1}
The solution belongs to $C^{\beta}(\overline{\Omega})$ for some $\beta$ ($0<\beta<1$) and satisfies
\begin{equation}\label{Log-form5}
\|u_{\varepsilon} \|_{C^{0,\beta}(\overline{\Omega})}\leq M,
\end{equation}
where $M$ and $\beta$ are constants independent of $\varepsilon$.
\end{prop}

\begin{proof}[\bf Proof.]
Since $\|u_\varepsilon \|_{W_0^{1,p}(\Omega)}<+\infty$ and $u_\varepsilon\geq 0$, applying in \cite[Lemma 10.8, p.271-271, p.277]{Book-Gilbarg-Trudinger} to \eqref{regul-prob}, we obtain the following estimate
\begin{equation}\label{Log-form6}
 \sup \{ |u_\varepsilon(x)|:\ x\in \Omega \}\leq c_1.
\end{equation}
By \cite[Theorem 1.1, p.251]{Lady-Ural1968}, we deduce that  $u\in C^{\beta}(\overline{\Omega})$ for some $(0<\beta<1)$, that is, it satisfies
\begin{equation}\label{Log-form7}
\mathop{\sup}_{x,y\in \overline{\Omega},x\neq y}\frac{|u_\varepsilon(x)-u_\varepsilon(y)|}{|x-y|^{\beta}}\leq c_2,
\end{equation}
where $\beta$ and $c_2$ are constants independent of $\varepsilon$. Combining \eqref{Log-form6} and \eqref{Log-form7}, we get \eqref{Log-form5}. This completes the proof.
\end{proof}

\begin{coro}\label{Coro1}
$u_\varepsilon \rightarrow u$ uniformly in $\overline{\Omega}$, as $\varepsilon\rightarrow 0$.
\end{coro}

\begin{proof}[\bf Proof.]
Using Arzela and Ascoli's theorem, Proposition \ref{weak-solution1} and Proposition \ref{holder-Regularity1}, the proof is completed.
\end{proof}

Before we apply Korevaar's concavity maximum principle to our problem, we need get more regularity of the solution $u_\varepsilon$ in compact subsets of $\Omega$ for small $\varepsilon >0$.

\begin{prop}\label{Regularity1}
For any $\delta>0$, if we choose numbers $\varepsilon_0>0$ and $\tau>0$ sufficiently small, we have the following: for any $\varepsilon$ $(0<\varepsilon<\varepsilon_0)$,
\begin{equation}
M \geq u_{\varepsilon} \geq \tau >0 \ \ \mathrm{in}\ \ \Omega_\delta,
\end{equation}
where $M$ is the constant in Proposition $\ref{holder-Regularity1}$ and $\Omega_\delta$ is the domain defined in Proposition $\ref{prop2}$.
\end{prop}
\begin{proof}[\bf Proof]
Combining Corollary \ref{Coro1} and the positivity of $u$, we can deduce the conclusion.
\end{proof}

\begin{prop}\label{Regularity2}
For any $\delta>0$ and $\varepsilon$ $(0<\varepsilon<\varepsilon_0)$, the solution $u_\varepsilon$ belongs to $C^{\infty}(\Omega_\delta)$ and satisfies
\begin{equation}
\| u_\varepsilon\|_{C^{1,\beta}(\overline{\Omega}_{2\delta})}\leq C,
\end{equation}
where $\beta$ $(0<\beta<1)$ and $C$ are constants independent of $\varepsilon>0$, and $\varepsilon_0$ is the number in Proposition .
\end{prop}
\begin{proof}[\bf Proof]
Since $u_\varepsilon$ is also a weak solution to the equation
\begin{equation}\label{PDE7}
\sum_{j=1}^{n} D_j a_j(x,u,\nabla u)+a(x,u,\nabla u)=0, \ \ \mathrm{in}\ \ \Omega_\delta,
\end{equation}
where for $i=1,\cdots,n$,
\begin{equation}
a_j(x,u,\nabla u)=e^{-|x|^2/2} (\varepsilon u^2+|\nabla u|^2)^{\frac{p-2}{2}}D_j u,
\end{equation}
and
\[
a(x,u,\nabla u)=\lambda^{\varepsilon}_{p,\gamma} |u|^{p-2}{u}e^{-|x|^2/2}-\varepsilon (\varepsilon u^2+|\nabla u|^2)^{\frac{p-2}{2}} ue^{-|x|^2/2}.
\]

We observe that
\begin{equation}
C_1 ({\sqrt{\varepsilon}+|\nabla u|})\leq (\varepsilon u^2+|\nabla u|^2)^{\frac{1}{2}}\leq C_2({\sqrt{\varepsilon}+|\nabla u|}),
\end{equation}
where $C_1=(1/{\sqrt{2}})\min\{(\tau/2),1 \}$ and $C_2=\max\{ 2M,1\}$. Therefore, we can easily checked the assumptions in \cite[Theorem 1]{Tolksdorf1984}, then we can use this theorem to \eqref{PDE7}, we deduce that $u_\varepsilon \in C^1(\Omega_\delta)$ and $\nabla u_\varepsilon$ is H$\mathrm{\ddot{o}}$lder continuous on $\Omega_\delta$. By regularity of the elliptic partial differential equations in \cite{Book-Gilbarg-Trudinger} (see also \cite[Chapter 4]{Lady-Ural1968}), we see that $u_\varepsilon \in C^{\infty}(\Omega_\delta)$. Furthermore, using Tolksdorf's interior estimate in \cite{Tolksdorf1984}, we get $u\in C^{1,\beta}(\overline{\Omega}_{2\delta})$, where $\beta$ is a constant $(0<\beta<1)$. Thus, this complete the proof.
\end{proof}

In the following, we define function
\[
\upsilon_\varepsilon=-\log u_\varepsilon,\ \mathrm{in}\ \Omega_\delta,
\]
and the concavity function to $c_\varepsilon(x,y,t)$ as follows:
\[
c_\varepsilon:=\upsilon_\varepsilon((1-t)x+ty)-(1-t)\upsilon_\varepsilon(x)-t\upsilon_\varepsilon(y),
\]
for $(x,y,t)\in \Omega_\delta\times \Omega_\delta \times [0,1]$. By proposition \ref{Regularity1}, $\upsilon_\varepsilon$ is well defined in $\Omega_\delta$, for small $\delta>0$. Note that $\upsilon_\varepsilon$ is convex if and only if $c_\varepsilon\leq 0$, for all $(x,y,t)\in \Omega_\delta\times \Omega_\delta \times [0,1]$.
\begin{prop}\label{Korevaar's-comcav-maxi-theo}
For any $\delta>0$ and any $\varepsilon$ $(0<\varepsilon<\varepsilon_0)$, the function $c_\varepsilon$ attains its positive maximum on the boundary $\partial \{ \Omega_\delta \times \Omega_\delta\} \times [0,1]$, provided it is anywhere positive.
\end{prop}

\begin{proof}[\bf Proof.]
By Proposition \ref{Regularity2}, we know that $\upsilon_\varepsilon \in C^\infty (\Omega_\delta)$, and satisfies the equation
\[
\sum_{i,j=1}^{n} a^{ij}(\nabla \upsilon_\varepsilon) D_{ij} \upsilon_\varepsilon-b(x,\nabla \upsilon)=0\ \ \mathrm{in}\ \Omega_\delta,
\]
where
\[
\sum_{i,j=1}^{n} a^{ij}(\nabla \upsilon_\varepsilon) D_{ij} \upsilon_\varepsilon=\mathrm{div}[(\varepsilon+|\nabla  \upsilon_\varepsilon|)^{\frac{p-2}{2}}\nabla \upsilon_\varepsilon],
\]
and
\[
b(x,\nabla \upsilon)=-\lambda^{\varepsilon}_{p,\gamma}+[\varepsilon-(p-1)|\nabla \upsilon_\varepsilon|^2+(x,\nabla \upsilon_\varepsilon)](\varepsilon+|\nabla \upsilon_\varepsilon|^2)^{\frac{p-2}{2}}.
\]
Thus, we apply Korevaar's  concavity maximum principle that is \cite[Theorem 1.3, p.604]{Korevaar1} to function $\upsilon_\varepsilon$ in the convex domain $\Omega_\delta$, this complete the proof.
\end{proof}

\begin{prop}\label{prop4}
For any small $\nu>0$, if we choose $\varepsilon_1>0$ sufficiently small, for any $\varepsilon$ with $0<\varepsilon<\varepsilon_1$, the function $\upsilon_\varepsilon=\ln u_\varepsilon$ is concave in $\Omega_\nu$.
\end{prop}

\begin{proof}[\bf Proof]
First of all, we can choose $\nu>0$ sufficiently small to get
\begin{equation}
\Omega \setminus \Omega_{2\nu}\subseteq N,
\end{equation}
where $N$ is the neighborhood of $\partial \Omega$ obtained in Proposition \ref{prop1}.

On the other hand, it follows from Proposition \ref{Regularity2} that for small $\varepsilon>0$
\begin{equation}\label{formula-1}
\|u_\varepsilon \|_{C^{1+\beta}(\overline{\Omega}_{\nu/2})}\leq C,
\end{equation}
where $\beta$ $(0<\beta<1)$ and $C$ are constants independent of $\varepsilon$. Then, since the embedding $C^{1+\beta}(\overline{\Omega}_{\nu/2})\hookrightarrow C^{1}(\overline{\Omega}_{\nu/2})$ is compact, we have
\begin{equation}\label{formula-2}
u_\varepsilon \rightarrow u\ \ \text{in}\quad C^1(\overline{\Omega}_{\nu/2})\ \mathrm{as}\ \varepsilon\rightarrow 0.
\end{equation}
By \eqref{formula-2} and Proposition \ref{prop1}, we have
\begin{equation}\label{formula-3}
|\nabla u_\varepsilon|\geq \frac{1}{2}\eta\quad \text{in}\ C^1(\overline{\Omega}_{\nu/2} \setminus \Omega_{2\nu}),
\end{equation}
for small $\varepsilon>0$. By \eqref{formula-1} and \eqref{formula-3}, we can choose the elliptic constant of \eqref{regul-prob} independently of $\varepsilon$, for small $\varepsilon>0$. Therefore, it  follows from Schauder estimates for elliptic partial differential equations in \cite{Book-Gilbarg-Trudinger} that for small $\varepsilon>0$, we have
\begin{equation}\label{formula-4}
\|u_\varepsilon \|_{C^{2+\beta}(\overline{\Omega}_{\nu/2}\setminus \Omega_{2\nu})}\leq C,
\end{equation}
where $C$ is a constant independent of $\varepsilon>0$.

Furthermore, using the compactness of the embedding
\[
C^{2+\beta}(\overline{\Omega}_{\nu/2}\setminus \Omega_{2\nu})\hookrightarrow C^{2}(\overline{\Omega}_{\nu/2}\setminus \Omega_{2\nu}),
\]
By Proposition \eqref{formula-3} and \ref{prop1}, we obtain the following facts: for small $\varepsilon>0$,

\emph{
The matrix $[-D_{ij} \upsilon_\varepsilon]$ is positive on $\overline{\Omega}_{\nu/2} \setminus \Omega_{2\nu}$ where $D_{ij}=\frac{\partial^2}{\partial x_i \partial x_j}$
}
and

\emph{Every tangent plane to\ the\ graph\ of\ $\upsilon_\varepsilon$ on $\partial \Omega_\nu$ lies above the graph on $\overline{\Omega}_\nu$ and contacts it only at tangent point in $\overline{\Omega}_\nu$.}

Finally, according to \cite[Lemma 2.1, p.609]{Korevaar1} and above fact, we see that the concavity function $c_\varepsilon$ does not attain its positive maximum on the boundary $\partial \{ \Omega_\nu \times \Omega_\nu\} \times [0,1]$ for small $\varepsilon>0$. Therefore, by Proposition \ref{Korevaar's-comcav-maxi-theo}, we deduce that $c_\varepsilon$ is non-positive for small $\varepsilon>0$. This completes the proof.
\end{proof}

\begin{prop}\label{prop3}
For any small $\delta>0$, the function $\upsilon=\ln u$ is concave in $\Omega\setminus \Omega_\delta$, where $\Omega_\delta$ is a strongly convex domain in Proposition $\ref{prop2}$.
\end{prop}

\begin{proof}[\bf Proof.]
By Proposition \ref{prop4} and Corollary \ref{Coro1}, we deduce the conclusion.
\end{proof}

\subsection{Log-concavity of the first eigenvalue to \eqref{PDE1}}

\begin{proof}[\bf Proof of Theorem \ref{T2}.]
First of all, we note that when $\Omega$ is a strong convex, by Proposition \ref{prop2} and Proposition \ref{prop3}, we conclude its proof.

When $\Omega$ is a bounded convex domain in $\R^n$ with boundary of class $C^2$, we can choose a sequence of strongly convex domain $\Omega_k$ satisfying
\[
\overline{\Omega}_k\subset \Omega_{k+1}\quad \mathrm{for\ all}\quad k\geq 1\quad \mathrm{and}\quad \mathop{\cup}_{k=1}^{\infty}\Omega_k=\Omega.
\]
Let $u_k\in W^{1,p}_{0}(\Omega_k,\gamma)$ be the unique positive solution to \eqref{PDE1} corresponding to $\Omega_k$. Without of generality, we may assume that $\| u_k\|_{L^{p}(\Omega,\gamma)}=1$. We can extend the function $u_k$ to a function in $\Omega$ by letting $u_k=0$ in $\Omega\setminus \Omega_k$. Then $u_k\in W^{1,p}_0(\Omega,\gamma)$. Since $\|u_k \|_{L^p(\Omega_k)}=\|u_k \|_{L^p(\Omega)}=1$, by the minimizing property of $u_k$ and the monotonically decreasing property of eigenvalue, we see that $\{u_k\}$ is bounded in $W^{1,p}_0(\Omega,\gamma)$. Therefore, there exists a subsequence $u_{k'}$ and a function $u\in W^{1,p}_0(\Omega,\gamma)$ satisfying
\[
u_{k'}\rightharpoonup u\quad \mathrm{weakly\ in}\ W^{1,p}_0(\Omega,\gamma)\ \mathrm{as}\ k'\rightarrow \infty,
\]
and
\[
u_{k'}\rightarrow u\quad \mathrm{a.e.\ in}\ \Omega.
\]
Since $u_k\geq 0$ in $\Omega$, we know that $u\geq0$. By the Sobolev embedding theorem, we have that $\|u\|_{L^p(\Omega)}=1$. By the minimiizing properties of $u_k$ and the lower semi-continuity of the norm of $W^{1,p}_0(\Omega,\gamma)$, we deduce that $u$ is a solution of \eqref{PDE1}. Then it follows from the uniqueness of the positive solution to \eqref{PDE1}, we deduce that $u$ is the unique solution to \eqref{PDE1}.

On the other hand, since the $C^{\beta}(\overline{\Omega})$-estimate of the solution to \eqref{PDE1} is independent of small smooth perturbation of the boundary $\partial \Omega$, we obtain the estimate
\[
\| u_k\|_{C^{\beta}(\overline{\Omega}_k)}\leq c,
\]
where $0<\beta<1$ and $c>0$ are constants independent of $k$.

Finally, according to Arzel\`{a} and Ascoli's theorem, we obtain
\[
u_k\rightarrow u\quad \mathrm{uniformly\ on\ any\ compact\ subset\ of}\ \Omega.
\]
Since Theorem \ref{T2} hold when $\Omega_k$ is strongly convex, that is, $u_k$ is log-concave, we deduce that $u$ is log-concave. This completes its proof.
\end{proof}

\section{BRUNN-MINKOWSKI TYPE INEQUALITY}\label{Se5}

In this section, we prove the Brunn-Minkowski-type inequality for the first eigenvalue to weighted $p$-Laplace equation.

Define the set
\[
\C=\left\{ x\in \Omega:\ \nabla u(x) =0\right\},
\]
where $u$ is a non-trivial solution of \eqref{PDE1} in $\Omega$. Thus, the weighted $p$-Laplace operator (applied to $u$) is uniformly elliptic on compact subsets of $\Omega\setminus \bar{\C}$. By a standard regularity results for solution of elliptic equations (see also Section \ref{Se3}), we know that $u\in C^2(\Omega\setminus \bar{\C})$; Moreover, the function $w:=-\ln u$ solves
\begin{equation}\label{PDE2}
\left\{
\begin{aligned}
& \Delta_p w-(x,\nabla w)|\nabla w|^{p-2}=\lambda_{p,\gamma}+(p-1)|\nabla w|^{p}, \ \ \ \ \ \  \mathrm {in} \ \ \Omega \setminus \bar{\C},  \\
& \mathop {\lim}_{x\rightarrow \partial \Omega} w=+\infty.
\end{aligned}
\right.
\end{equation}

\begin{lemm}[Lemma 2.1 in \cite{Andrea-Paolo2006}]\label{L1}
For $i=0,1$, let $\Omega_i$ be an open, bounded, convex set and $u_i\in C^1(\Omega_i)$ be a strictly concave function such that $ \mathop {\lim}_{x\rightarrow \partial \Omega_i} u_i=-\infty$. Then, for $t\in[0,1]$, $\tilde{u}\in C^1(\Omega_t)$ and it is strictly concave; moreover, for every $z\in \Omega_t$, there exists a unique couple of points $(x,y)\in \Omega_0\times \Omega_1$ such that
\begin{equation}\nonumber
  z=(1-t)x+ty,
\end{equation}
\begin{equation}\nonumber
  \tilde{u}(z)=(1-t)u_0(x)+tu_1(y),
\end{equation}
\begin{equation}\nonumber
  \nabla \tilde{u}(z)=\nabla u_0(x)=\nabla u_1(y).
\end{equation}
If in addition $z\in\Omega_t$ is such that $u_0$ and $u_1$ are twice differentiable at corresponding point $x$ and $y$ respectively, and $D^2u(x_0), D^2u_1(y)<0$, then $\tilde{u}$ is twice differentiable at $x$ and
\begin{equation}\nonumber
D^2 \tilde{u}(z)=\left[  (1-t)( D^2\tilde{u}(x) )^{-1}+t( D^2 \tilde{u}_1(y))^{-1}\right]^{-1}.
\end{equation}
\end{lemm}

We remark that in the proof of \cite[Lemma 2.1,p.49]{Andrea-Paolo2006} (see \cite[formula (12), p.50]{Andrea-Paolo2006}), we have the following conclusion, that is
\begin{equation}\label{BMI-form1}
(-\tilde{u})^*=(1-t)(-u_0)^*+t(-u_1)^*,
\end{equation}
where the symbol $*$ means the usual conjugation of convex functions (see \cite[section 12]{Rocka1970}).

\begin{lemm}[Lemma 2.2 in \cite{Andrea-Paolo2006}]\label{L2}
Let $\Omega_0$ and $\Omega_1$ be two open bounded convex sets and let $u_i:\Omega_i\rightarrow \R$, $u_i\in C^1(\Omega_i)$, for $i=0,1$, be a strictly concave function. Fix $t\in[0,1]$, for a point $z\in \Omega_t$, let $(x,y)\in \Omega_0\times \Omega_1$ be a unique couple of points determined by Lemma \ref{L1}. If $\nabla \tilde{u}(z)\neq 0$, and $u_0$ and $u_1$ are twice differentiable at $x$ and $y$ respectively, with $D^2 u_0(x)$, $D^2 u_1<0$, then
\begin{equation}\nonumber
\Delta_p \tilde{u}(z)\geq (1-t)\Delta_p u_0(x)+t\Delta_p u_1(y).
\end{equation}
\end{lemm}

\begin{proof}[\bf Proof of Theorem \ref{T1}.]
Let $\Omega_0$ and $\Omega_1$ be open and bounded. Let $u_i$ be the solution of problem, for $i=0,1$. For simplicity, set
\[
\lambda_i=\lambda_{p,\gamma}(\Omega_i)\ \ \mathrm{for}\ \ i=0,1,\ \ \lambda_t=(1-t)\lambda_0+t\lambda_1.
\]
Set $w_i=-\ln u_i$ for $i=0,1$. By Theorem \ref{T2}, we have $w_i$ is convex in $\Omega_i$, it belongs to $C^2(\Omega_i\setminus \bar{\C}_i)$ and it is a solution of \eqref{PDE2} in $\Omega_i \setminus \bar{\C}_i$, where $\C_i=\{x\in\Omega_i:\ \nabla u_i(x)=0 \}$ for $i=0,1$. Given $x\in \Omega_t$, we define
\begin{equation}
\tilde{w}(x)=\inf \left\{ (1-t)w_0(x_0)+tw_1(x_1):\ x_i \in \overline{\Omega}_i\ \text{for}\ i=0,1, x=(1-t)x_0+tx_1\right\}.
\end{equation}
From \cite[Corollaries 26.3.2 and 25.5.1]{Rocka1970}, we know that $\tilde{w}(x)\in C^1(\Omega_t)$.

Given $\varepsilon >0$, we define
\[
w_{i,\varepsilon}=w_i(x)+\varepsilon \frac{x^2}{2},\ \ \ x\in \Omega_i.
\]
The function ${w}_{i,\varepsilon}$ is strictly convex function in $\Omega_i$ and
\begin{equation}\label{BMI-form2}
{w}_{i,\varepsilon}\in C^2(\Omega_i \setminus \bar{\C}_i),\quad\  \mathrm{for}\ i=0,1.
\end{equation}
We consider the inf-convolution $\tilde{w}_\varepsilon$ of $w_{0,\varepsilon}$, $w_{1,\varepsilon}$, that is,
\begin{equation}\label{BMI-form3}
\tilde{w}_{\varepsilon}=\inf \left\{ (1-t)w_{0,\varepsilon}(x_0)+tw_{1,\varepsilon}(x_1):\ x_i \in \overline{\Omega}_i\ \text{for}\ i=0,1, x=(1-t)x_0+tx_1\right\}.
\end{equation}
It is clear that $w_{i,\varepsilon}$ converges uniformly to $w_i$ in $\Omega_i$ for $i=0,1$. Similar to \cite[p.53]{Andrea-Paolo2006}, we can also see that $\tilde{w}_{\varepsilon}$ converges uniformly to $\tilde{w}$ in $\Omega_t$, that is,
\begin{equation}\label{BMI-form4}
\nabla \tilde{w}_{\varepsilon}\ \mathrm{converges\ uniformly\ to }\ \nabla \tilde{w}\ \mathrm{on\ every\ compact\ subset\ of\  } \Omega_t.
\end{equation}
Actually, by \cite[Theorem 25.7]{Rocka1970}, we conclude that
\begin{equation}\label{BMI-form5}
\nabla \tilde{w}_{i,\varepsilon}\ \mathrm{converges\ uniformly\ to }\ \nabla \tilde{w_i}\ \mathrm{on\ every\ compact\ subset\ of\  } \Omega_i,
\end{equation}
for $i=0,1$, and
\begin{equation}\label{BMI-form6}
\nabla \tilde{w}_{\varepsilon}\ \mathrm{converges\ uniformly\ to }\ \nabla \tilde{w}\ \mathrm{on\ every\ compact\ subset\ of\  } \Omega_t,
\end{equation}

In the following, we express the $p$-Laplacian of $w_{i,\varepsilon}$ in terms of the $p$-Laplacian of $w_i$, that is, for $i=0,1$ and $x\in \Omega_i \setminus \bar{\C}_i$, put
\[
n_{i,\varepsilon}(x)=\frac{\nabla w_{i,\varepsilon}(x)}{|\nabla w_{i,\varepsilon}(x)|}.
\]
According to \cite[p.53-54]{Andrea-Paolo2006}, we have
\begin{equation}\label{BMI-form7}
\begin{aligned}
\Delta_p w_{i,\varepsilon}(x)&=\Delta_p w_i(x)+(|\nabla w_i(x)-\varepsilon x|^{p-2}-|\nabla w_i(x)|^{p-2})\Delta w_i(x)\\
&\quad + (p-2)\left[ |\nabla w_i(x)-\varepsilon x|^{p-2} \langle D^2 w_i(x)n_{i,\varepsilon}, n_{i,\varepsilon} \rangle -|\nabla w_i(x)|^{p-2}  \langle D^2 w_i(x)n_{i}, n_{i} \rangle \right]\\
&\quad -\varepsilon(n+p-2)|\nabla w_i(x)-\varepsilon x|^{p-2}.
\end{aligned}
\end{equation}
Let
\[
\C_t=\left\{ x\in \Omega_t:\ \nabla  \tilde{w} =0\right\},
\]
By \cite[Theorem 26.5 and 23.8]{Rocka1970} and \eqref{BMI-form1}, we have that
\begin{equation}\label{BMI-form8}
\begin{aligned}
\C_t&=\partial(-\tilde{w}^{*}(0))=\partial((1-t)(-w^{*}_0)+t(-w^{*}_1))\\
&=(1-t)\partial(-w^{*}_0)+t\partial(-w^{*}_1) =(1-t)\C_0+t\C_1,
\end{aligned}
\end{equation}
where the last equality, we used the fact that $\C_i=\{x\in \Omega_i:\nabla u_i=0\}=\{x\in \Omega_i:\nabla w_i=0\}$.
By Lemma \ref{L1} and \eqref{BMI-form2}, we have
\[
\tilde{w}_\varepsilon\in C^2(\Omega_t \setminus \bar{\C}_t).
\]
Moreover, for a fixed $z\in \Omega_t\setminus \bar{\C}_t$, by Lemma \ref{L1}, there exists a unique $(x_\varepsilon,y_\varepsilon)\in \Omega_0\times \Omega_1$, depending on $z$, such that $z=(1-t)x+ty$ and
\begin{equation}\label{BMI-form9}
\nabla \tilde{w}_\varepsilon (z)=\nabla {w}_{0,\varepsilon}(x_\varepsilon)=\nabla {w}_{1,\varepsilon}(y_\varepsilon).
\end{equation}
Since $\nabla \tilde{w}(z)\neq 0$, we have
\[
\nabla \tilde{w}_{0,\varepsilon}(x_\varepsilon)=\nabla \tilde{w}_{1,\varepsilon}(y_\varepsilon)\neq 0.
\]
By \eqref{BMI-form5} and \eqref{BMI-form6}, we can deduce that there exists $\varepsilon_1>0$ such that, for $0<\varepsilon<\varepsilon_1$,
\[
(x_\varepsilon,y_\varepsilon)\in \left( (\Omega_0\setminus \bar{\C}_0)\times (\Omega_1\setminus \bar{\C}_1) \right).
\]
Then, using Lemma \ref{L2}, we have
\[
\Delta_p \tilde{w}_{\varepsilon}(z)\leq (1-t) \Delta_p {w}_{0,\varepsilon}(x_\varepsilon)+t\Delta_p {w}_{1,\varepsilon}(y_\varepsilon).
\]
Thus, by \eqref{BMI-form7}, we have
\[
\Delta_p \tilde{w}_{\varepsilon}(z) \leq (1-t)\Delta_p w_{0}(x_\varepsilon)+t\Delta_p {w}_{1}(y_\varepsilon)+F_\varepsilon(z),
\]
where
\begin{align*}
{F}_\varepsilon(z)&=(1-t)\Delta w_0(x_\varepsilon)(|\nabla w_0(x_\varepsilon)-\varepsilon x_\varepsilon|^{p-2}-|\nabla w_0(x_\varepsilon)|^{p-2})+ (1-t)(p-2)\\
&\quad \times\left[ |\nabla w_0(x_\varepsilon)-\varepsilon x_\varepsilon|^{p-2} \langle D^2 w_0(x_\varepsilon)n_{0,\varepsilon}, n_{0,\varepsilon} \rangle -|\nabla w_0(x_\varepsilon)|^{p-2}  \langle D^2 w_0(x_\varepsilon)n_{0,\varepsilon}, n_{0,\varepsilon} \rangle \right]\\
&\quad +t\Delta w_1(y_\varepsilon)(|\nabla w_1(y_\varepsilon)-\varepsilon y_\varepsilon|^{p-2}-|\nabla w_0(y_\varepsilon)|^{p-2})+ t(p-2)\\
&\quad \times\left[ |\nabla w_1(y_\varepsilon)-\varepsilon y_\varepsilon|^{p-2} \langle D^2 w_1(y_\varepsilon)n_{1,\varepsilon}, n_{1,\varepsilon} \rangle -|\nabla w_1(y_\varepsilon)|^{p-2}  \langle D^2 w_1(y_\varepsilon)n_{1,\varepsilon}, n_{1,\varepsilon} \rangle \right]\\
&\quad -\varepsilon(n+p-2)\left[(1-t) |\nabla w_0(x_\varepsilon)-\varepsilon x_\varepsilon|^{p-2}+t |\nabla w_1(y_\varepsilon)-\varepsilon y_\varepsilon|^{p-2}\right].
\end{align*}
Then, by \eqref{PDE2}, we can deduce that
\begin{align*}
\Delta_p \tilde{w}_{\varepsilon}(z)& \leq (1-t)\lambda_0+t\lambda_1+(1-t)(p-1)|\nabla w_0(x_\varepsilon)|^p+t(p-1)|\nabla w_1(y_\varepsilon)|^p\\
&\quad +(1-t)(x_\varepsilon,\nabla w_0(x_\varepsilon))+t(y_\varepsilon,\nabla w_1(y_\varepsilon))+{F}_\varepsilon(z).
\end{align*}
By the definition of $w_{0,\varepsilon}$, $w_{1,\varepsilon}$ and \eqref{BMI-form9}, we can write
\[
\nabla w_0(x_\varepsilon)=\nabla \tilde{w}_\varepsilon(z)+\varepsilon x_\varepsilon,\ \ \nabla {w}_1(y_\varepsilon)=\nabla \tilde{w}_\varepsilon(z)+\varepsilon y_\varepsilon.
\]
Thus, we have
\begin{equation}\label{BMI-form10}
\Delta_p \tilde{w}_{\varepsilon}(z) \leq (1-t)\lambda_0+t\lambda_1+(p-1)|\nabla \tilde{w}_{\varepsilon}(z)|^p  +(z,\nabla \tilde{w}_{\varepsilon}(z))+\tilde{F}_\varepsilon(z),
\end{equation}
where
\begin{align*}
\tilde{F}_\varepsilon(z)&={F}_\varepsilon(z)+\varepsilon[(1-t)|x_\varepsilon|^2+t|y_\varepsilon|^2] +(p-1)\\
&\quad \times [|\nabla \tilde{w}_\varepsilon(z)|^p-(1-t)|\nabla \tilde{w}_\varepsilon(z)+\varepsilon x_\varepsilon|^p-t|\nabla \tilde{w}_\varepsilon(z)+\varepsilon y_\varepsilon|^p].
\end{align*}
It is clear that
\[
\tilde{F}_\varepsilon\rightarrow 0 \ \ \mathrm{pointwise \ in}\ \Omega_t \setminus \bar{\C}_t.
\]
Moreover, if $T$ is a compact subset of $\Omega_t\setminus \bar{\C}_t$, then there exist $\overline{\varepsilon}=\overline{\varepsilon}(T)$ and two compact subset of $A_0$ and $A_1$ of $\Omega_0\setminus \bar{\C}_0$ and $\Omega_1\setminus \bar{\C}_1$ respectively, such that. for every point $z\in T$ and $0<\varepsilon<\overline{\varepsilon}$,
\[
(x_\varepsilon,y_\varepsilon)\in A_0 \times A_1,
\]
it follows from \eqref{BMI-form9} and the definition of $\C_i$ for $i=0,1$. Since $w_i\in C^2(\Omega_i\setminus \bar{\C}_i)$, for $i=0,1$, we deduce that $|\tilde{F}_\varepsilon|$ is uniformly bounded on $T$ w.r.t $0<\varepsilon<\overline{\varepsilon}$. Consequently, the sequence $\tilde{F}_\varepsilon$ is uniformly bounded on compact subset of $\Omega_t \setminus \bar{\C}_t$.

Set $\tilde{u}_\varepsilon(z)=e^{-\tilde{w}_\varepsilon(z)}$, $z\in \Omega_t$, $\varepsilon>0$. Since $\tilde{w}_\varepsilon \rightarrow +\infty$ on the boundary, we deduce that $\tilde{u}_\varepsilon\in C(\overline{\Omega}_t)$ and it vanished on $\partial \Omega_t$. Moreover, by Lemma \ref{L1} and \eqref{BMI-form9}, for every $z\in \Omega_t$, we have
\begin{align*}
|\nabla \tilde{u}_\varepsilon(z)|&=|\nabla \tilde{w}_\varepsilon(z)|e^{-\tilde{w}_\varepsilon(z)}=|\nabla {w}_{0,\varepsilon}|^{1-t}|\nabla {w}_{1,\varepsilon}|^{t}e^{-(1-t){w}_{0,\varepsilon}-t{w}_{1,\varepsilon}}\\
&=[|\nabla {w}_{0,\varepsilon}| e^{-{w}_{0,\varepsilon}}]^{1-t}[|\nabla {w}_{1,\varepsilon}| e^{-{w}_{1,\varepsilon}}]^{t}\\
&=|\nabla ( e^{-{w}_{0,\varepsilon}(z)})|^{1-t} |\nabla ( e^{-{w}_{1,\varepsilon}(z)})|^{t}.
\end{align*}
From the last equality, the definition of $w_{0,\varepsilon}$ and $w_{1,\varepsilon}$ and the regularity of $w_0$ and $w_1$ (see Section 3), we obtain that $|\nabla\tilde{u}_\varepsilon|$ is bounded in $\overline{\Omega}_t$, and, in particular, we have $\tilde{u}_\varepsilon \in W^{1,p}_0(\Omega_t,\gamma)$. By \eqref{BMI-form10}, for $z\in \Omega_t\setminus \bar{\C}_t$, we have
\begin{align*}
\Delta_p \tilde{u}_{\varepsilon}(z)& \geq- [(1-t)\lambda_0+t\lambda_1]|\tilde{u}_\varepsilon|^{p-2}u_\varepsilon+(z,\nabla \tilde{u}_\varepsilon)|\nabla \tilde{u}_{\varepsilon}(z)|^{p-2}-\tilde{F}_\varepsilon(z)|\tilde{u}_\varepsilon(z)|^{p-2}\tilde{u}_\varepsilon(z).
\end{align*}
Thus, Multiply the above inequality by $-\tilde{u}_\varepsilon$ to get
\begin{equation}\label{BMI-form11}
-\tilde{u}_\varepsilon\Delta_p \tilde{u}_\varepsilon+(z,\nabla \tilde{u}_\varepsilon)|\nabla \tilde{u}_\varepsilon|^{p-1}\leq \lambda_t |\tilde{u}_\varepsilon|^{p}+\tilde{F}_\varepsilon(z)|\tilde{u}_\varepsilon(z)|^{p}.
\end{equation}

In the following, we consider a sequence of compact sets $T_j=\overline{A}_j\setminus B_j$, where $A_j$ and $B_j$ are open convex sets so that $A_j\subset\subset \Omega_t$, $B_j\supset \C_t$ and $\overline{B}_j\subset A_j$. Moreover, assume that $A_j \rightarrow \Omega_t$ and $B_j  \rightarrow \C_t$ in the Hausdorff metric, as $j \rightarrow +\infty$. For every $j\in \mathbb{N}$, we integrate it over $T_j$, we have
\[
\int_{T_j} -\tilde{u}_\varepsilon\Delta_p \tilde{u}_\varepsilon+(z,\nabla \tilde{u}_\varepsilon)|\nabla \tilde{u}_\varepsilon|^{p} d\gamma(z) \leq \lambda_t \int_{T_j} |\tilde{u}_\varepsilon|^{p} d\gamma(z).
\]
We integrate it over $T_j$, using integrating by parts, we have
\begin{equation}\label{BMI-form12}
\int_{T_j} |\nabla \tilde{u}_\varepsilon|^{p} d\gamma(z)-\int_{\partial T_j} \tilde{u}_\varepsilon | \nabla \tilde{u}_\varepsilon |^{p-2} \frac{\partial \tilde{u}_\varepsilon}{\partial n}d\gamma_{\partial T_j}\leq \lambda_t \int_{T_j} |\tilde{u}_\varepsilon|^{p} d\gamma(z)+\int_{T_j} \tilde{F}_\varepsilon(z)\tilde{u}^p_\varepsilon d\gamma(z).
\end{equation}
Note that, since $\tilde{w}_\varepsilon$ converges to $\tilde{w}$ uniformly, then $\tilde{u}_\varepsilon$ converges to $\tilde{u}=e^{-\tilde{w}}$ uniformly in $\overline{\Omega}_t$.  In particular, we have $\tilde{u}\in W_0^{1,p}(\Omega_t,\gamma)$. Moreover, by \eqref{BMI-form5} and \eqref{BMI-form5}, we know that $\nabla\tilde{u}_\varepsilon$ converges to $\nabla\tilde{u}$ uniformly in $T_j$. Passing to the limit for $\varepsilon\rightarrow 0$ in \eqref{BMI-form12}, using the uniform convergence of $\tilde{u}_\varepsilon$ and $\nabla \tilde{u}_\varepsilon$, the properties of $\tilde{F}_\varepsilon$ and the Dominated convergence Theorem, we have
\[
\int_{T_j} |\nabla \tilde{u}|^{p} d\gamma(z)-\int_{\partial T_j} \tilde{u} | \nabla \tilde{u}|^{p-2} \frac{\partial \tilde{u}}{\partial n}d\gamma_{\partial T_j}\leq \lambda_t \int_{T_j} |\tilde{u}|^{p} d\gamma(z)\leq \lambda_t \int_{\Omega_t} |\tilde{u}|^{p} d\gamma(z).
\]
We can rewrite the last inequality in the following form
\[
\int_{A_j\setminus B_j} |\nabla \tilde{u}|^{p} d\gamma-\int_{\partial A_j \cup \partial B_j} \tilde{u}| \nabla \tilde{u}|^{p-2} \frac{\partial \tilde{u}}{\partial n}d\gamma_{\partial A_j \cup \partial B_j}\leq \lambda_t \int_{\Omega_t} |\tilde{u}|^{p} d\gamma(z).
\]
Since $\tilde{u}\rightarrow 0$ as $z\rightarrow \partial \Omega_t$ and $\nabla \tilde{u}\rightarrow 0$ as $z\rightarrow \partial \C_t$ (recall that $\C_t=\{ x\in\Omega_t:\ \nabla \tilde{w}(x)=0\}=\{ x\in\Omega_t:\ \nabla \tilde{u}(x)=0\}$). Passing to the limit for $j\rightarrow +\infty$, we have
\[
\int_{\Omega_t\setminus \C_t} |\nabla \tilde{u}|^{p} d\gamma(z) \leq \lambda_t \int_{\Omega_t} |\tilde{u}|^{p} d\gamma(z).
\]
Since
\[
\int_{\Omega_t\setminus \C_t} |\nabla \tilde{u} |^pd\gamma(z) =\int_{\Omega_t } |\nabla \tilde{u}|^p d\gamma(z),
\]
we can deduce that
\begin{equation}\nonumber
\lambda_{p,\gamma}(\Omega_t)\leq \lambda_t=(1-t)\lambda_0 +t\lambda_1.
\end{equation}
This completes the proof.
\end{proof}

\section*{Acknowledgement} The author would like to thank Prof. Andrea Colesanti and Paolo Salani for their patient guidance and warm encouragement, and the Department of Mathematics and Computer Science of the University of Florence for the hospitality. The author would like to thank China Scholarship Council (CSC) for the financial support during visit in the University of Florence.

\bibliographystyle{amsplain}

\end{document}